\documentclass[twoside,12pt]{article}

\usepackage{times,amsmath,amssymb,amsthm,cite}

\newcommand{\subj}[1]{\par\noindent{\bf AMS Subject Classifications: }#1.}
\newcommand{\keyw}[1]{\par\noindent{\bf Keywords: }#1.}

\numberwithin{equation}{section}
\numberwithin{figure}{section}

\newtheorem{theorem}{Theorem}[section]
\newtheorem{lemma}[theorem]{Lemma}

\theoremstyle{definition}
\newtheorem{definition}[theorem]{Definition}

\theoremstyle{remark}
\newtheorem{remark}[theorem]{Remark}

\everymath{\displaystyle}
\date{}
\newcommand{\ijde}
{\vspace{-1in}\normalsize\flushleft
This is a preprint of a paper whose final and definite form will be published in:\\
Int. J. Difference Equ. ({\tt http://campus.mst.edu/ijde}).\\
Submitted Aug 24, 2012; Revised Nov 19, 2012; Accepted Nov 19, 2012.\\\vspace{1mm}\hrule\vspace{5mm}}

\begin{document}

\title{\ijde \center\Large\bf A Non-Differentiable
Quantum Variational Embedding in Presence of Time Delays}
\author{{\bf Gast\~{a}o S. F. Frederico${}^{a, b}$}
and {\bf Delfim F. M. Torres${}^{b}$}\\
{\tt \{gastao.frederico, delfim\}@ua.pt}\\[0.3cm]
${}^{a}$~Greg\'{o}rio Semedo University, Luanda, Angola\\[0.3cm]
${}^{b}$~Center for Research and Development in Mathematics and Applications\\
Department of Mathematics, University of Aveiro\\
3810-193 Aveiro, Portugal}
\maketitle
\thispagestyle{empty}


\begin{abstract}
We develop Cresson's non-differentiable embedding to quantum
problems of the calculus of variations and optimal control with
time delay. Main results show that the dynamics of non-differentiable
Lagrangian and Hamiltonian systems with time delays can be
determined, in a coherent way, by the least-action principle.
\end{abstract}


\subj{49K05, 49S05, 26A24}

\keyw{non-differentiability, scale calculus of variations,
scale optimal control, Euler--Lagrange equations,
embedding, coherence, time delay, Hamiltonian systems}


\section{Introduction}

Lagrangian systems play a fundamental role describing
motion in mechanics. The importance of such systems
is related to the fact that they can be derived
via the least-action principle using differentiable manifolds \cite{ar}.
Nevertheless, some important physical systems involve functions
that are non-differentiable.

A non-differentiable calculus was introduced in 1992 by Nottale
\cite{Nottale:1992,Nottale:1999}. A rigorous foundation to
Nottale's scale-relativity theory was recently given by
Cresson \cite{CD:Cresson:2005,CD:Cresson:2006,Cresson:2011}.
The calculus of variations developed in \cite{Cresson:2011}
cover sets of non-differentiable curves,
by substituting the classical derivative by
a new complex operator, known as the \emph{scale derivative}.
In \cite{alto,alto1} Almeida and Torres obtain several
Euler--Lagrange equations for variational functionals and
isoperimetric problems of the calculus of variations
defined on a set of H\"{o}lder curves.

The embedding procedure, introduced for stochastic processes in \cite{MR2337677},
can always be applied to Lagrangian systems \cite{MyID:228,MyID:226}.
In this work we prove that the embedding of Lagrangian and Hamiltonian
systems with time delays, via the least-action principle,
respect the principle of coherence. For the importance of
variational and control systems with delays we refer the reader to
\cite{MyID:231} and references therein.

The article is organized as follows. A brief review of the quantum calculus of
\cite{Cresson:2011}, which extends the classical differential
calculus to non-differentiable functions, is given in Section~\ref{sec:2}.
In Section~\ref{sec:cNT} we discuss the non-differentiable embedding within the time delay formalism:
Section~\ref{sec:CV} is devoted to the development of the non-differentiable embedding to variational
problems with time delay, where a causal and coherent embedding is obtained
by restricting the set of variations; in Section~\ref{sec:OP}
we prove that the non-differentiable embedding
of the corresponding Hamiltonian formalism is also coherent.


\section{The Quantum Calculus of Cresson}
\label{sec:2}

Let $\mathbb{X}^d$ denote the set $\mathbb{R}^{d}$ or $\mathbb{C}^{d}$,
$d \in \mathbb{N}$, and $I$ be an open set in $\mathbb{R}$
with $[t_1,t_2]\subset I$, $t_1<t_2$.
We denote by $\mathcal{G}\left(I,\mathbb{X}^d\right)$
the set of functions $f:I \rightarrow \mathbb{X}^d$
and by $\mathcal{C}^{0}\left(I,\mathbb{X}^d\right)$
the subset of functions of $\mathcal{G}\left(I,\mathbb{X}^d\right)$
that are continuous.

\begin{definition}[The $\epsilon$-left and $\epsilon$-right quantum derivatives]
Let $f\in \mathcal{C}^{0}\left(I, \mathbb{R}^{d}\right)$. For all
$\epsilon>0$, the $\epsilon$-left and $\epsilon$-right quantum derivatives of
$f$, denoted respectively by $\Delta_{\epsilon}^{-}f$
and $\Delta_{\epsilon}^{+}f$, are defined by
\begin{equation*}
\Delta_{\epsilon}^{-}f(t)=\frac{f(t)-f(t-\epsilon)}{\epsilon}
\quad \text{ and } \quad
\Delta_{\epsilon}^{+}f(t)=\frac{f(t+\epsilon)-f(t)}{\epsilon} \, .
\end{equation*}
\end{definition}

\begin{remark}
The $\epsilon$-left and $\epsilon$-right quantum derivatives
of a continuous function $f$ correspond to the classical derivative
of the $\epsilon$-mean function $f_{\epsilon}^{\sigma}$ defined by
\begin{equation*}
f_{\epsilon}^{\sigma}(t)=\frac{\sigma}{\epsilon}\int_{t}^{t+\sigma\epsilon}f(s)ds\, ,
\quad \sigma=\pm \, .
\end{equation*}
\end{remark}

Next we define an operator which generalize the classical
derivative.

\begin{definition}[The $\epsilon$-scale derivative]
\label{def:qd}
Let $f\in \mathcal{C}^{0}\left(I,\mathbb{R}^{d}\right)$.
For all $\epsilon>0$, the $\epsilon$-scale
derivative of $f$, denoted by $\frac{\square_{\epsilon}f}{\square t}$,
is defined by
\begin{gather*}
\frac{\square_{\epsilon}f}{\square t}
=\frac{1}{2}\left[\left(\Delta_{\epsilon}^{+}f
+\Delta_{\epsilon}^{-}f\right)
-i\left(\Delta_{\epsilon}^{+}f
-\Delta_{\epsilon}^{-}f\right)\right],
\end{gather*}
where $i$ is the imaginary unit.
\end{definition}

\begin{remark}
If $f$ is differentiable, we can take the limit of the scale
derivative when $\epsilon$ goes to zero. We then obtain the
classical derivative $\frac{df}{dt}$ of $f$.
\end{remark}

We also need to extend the scale derivative to complex valued
functions.

\begin{definition}
Let  $f\in \mathcal{C}^{0}\left(I,\mathbb{C}^{d}\right)$
be a continuous complex valued function.
For all $\epsilon>0$, the $\epsilon$ scale derivative of $f$,
denoted by $\frac{\square_{\epsilon}f}{\square t}$, is defined by
\begin{gather*}
\frac{{\square}_{\epsilon}f}{{\square}t}
=\frac{{\square}_{\epsilon}\textrm{Re}(f)}{\square t}
+i\frac{\square_{\epsilon}\textrm{Im}(f)}{\square t} \, ,
\end{gather*}
where $\textrm{Re}(f)$ and $\textrm{Im}(f)$ denote the real and
imaginary part of $f$, respectively.
\end{definition}

In Definition~\ref{def:qd}, the $\epsilon$-scale derivative
depends on $\epsilon$, which is a free parameter related to the
smoothing order of the function. This brings many difficulties in
applications to Physics, when one is interested in particular
equations that do not depend on an extra parameter. To solve these
problems, the authors of \cite{Cresson:2011} introduced a procedure
to extract information independent of $\epsilon$ but related
with the mean behavior of the function.

\begin{definition}
Let ${C^0_{conv}}\left(I\times
(0,1),\mathbb{R}^{d}\right)\subseteq {C^0}\left(I\times
(0,1),\mathbb{R}^{d}\right)$ be such that for any function $f \in
{C^0_{conv}}\left(I\times (0,1),\mathbb{R}^{d}\right)$ the
$\lim_{\epsilon\to 0}f(t,\epsilon)$
exists for any $t\in I$; and $E$ be a complementary of
${C^0_{conv}}\left(I\times (0,1),\mathbb{R}^{d}\right)$ in
${C^0}\left(I\times (0,1),\mathbb{R}^{d}\right)$. We define
the projection map $\pi$ by
$$
\begin{array}{lcll}
\pi: & {C^0_{conv}}\left(I\times (0,1),\mathbb{R}^{d}\right)
\oplus E & \to & {C^0_{conv}}(I\times \left(0,1),\mathbb{R}^{d}\right)\\
& f_{conv}+f_E  & \mapsto & f_{conv}
\end{array}
$$
and the operator $\left< \cdot \right>$ by
$$
\begin{array}{lcll}
\left< \cdot \right>: & {C^0}\left(I\times (0,1),\mathbb{R}^{d}\right)
& \to & {C^0}\left(I,\mathbb{R}^{d}\right)\\
& f & \mapsto & \left< f \right>: t\mapsto
\displaystyle\lim_{\epsilon\to 0}\pi(f)(t,\epsilon)\,.
\end{array}
$$
\end{definition}

We now introduce the quantum derivative of $f$ without
the dependence of $\epsilon$ \cite{Cresson:2011}.

\begin{definition}
\label{def:ourHD}
The quantum derivative of $f$ in the space
$\mathcal{C}^{0}\left(I, \mathbb{R}^{d}\right)$
is given by the rule
\begin{equation}
\label{eq:scaleDer}
\frac{\Box f}{\Box t}=\left<
\frac{{\Box_{\epsilon}}f}{\Box t} \right>.
\end{equation}
\end{definition}

The scale derivative \eqref{eq:scaleDer} has some nice properties.
Namely, it satisfies a Leibniz and a Barrow rule. First let us
recall the definition of an $\alpha$-H\"{o}lderian function.

\begin{definition}[H\"{o}lderian function of exponent $\alpha$]
Let $f\in C^0\left(I, \mathbb{R}^{d}\right)$. We say that $f$ is
$\alpha$-H\"{o}lderian, $0<\alpha<1$, if for all $\epsilon>0$ and
all $t$, $t'\in I$ there exists a constant $c>0$ such that
$|t-t'|\leqslant\epsilon$ implies
$\|f(t)-f(t')\|\leqslant c\epsilon^{\alpha}$,
where $\|\cdot\|$ is a norm in $\mathbb{R}^{d}$.
The set of  H\"{o}lderian functions of H\"{o}lder exponent $\alpha$
is denoted by $H^\alpha(I,\mathbb{R}^{d})$.
\end{definition}

In what follows, we will frequently use $\square$ to denote the
scale derivative operator $\frac{{\square}}{{\square}t}$.

\begin{theorem}[The quantum Leibniz rule \cite{Cresson:2011}]
\label{theo:mult} For $f\in H^\alpha\left(I,
\mathbb{R}^{d}\right)$ and $g\in H^\beta\left(I,
\mathbb{R}^{d}\right)$, with $\alpha+\beta>1$, one has
\begin{equation}
\label{eq:mult}
\square(f\cdot g)(t)=\square f(t) \cdot
g(t)+f(t)\cdot\square g(t)\,.
\end{equation}
\end{theorem}

\begin{remark}
For $f\in \mathcal{C}^1\left(I, \mathbb{R}^{d}\right)$ and $g\in
\mathcal{C}^1\left(I, \mathbb{R}^{d}\right)$, one obtains from
\eqref{eq:mult} the classical Leibniz rule $(f\cdot g)'=f'\cdot
g+f\cdot g'$. For convenience of notation, we sometimes write
\eqref{eq:mult} as $(f\cdot g)^\square(t)= f^\square(t)
\cdot g(t)+f(t)\cdot g^\square(t)$.
\end{remark}

\begin{theorem}[The quantum Barrow rule \cite{Cresson:2011}]
\label{Barrow} Let $f\in \mathcal{C}^0([t_1,t_2],\mathbb{R})$ be
such that $\Box f / \Box t$ is continuous and
\begin{equation}
\label{nec_condition}
\lim_{\epsilon\to0} \int_{t_{1}}^{t_{2}}
\left(\frac{\Box_\epsilon f}{\Box t}\right)_E(t)dt=0.
\end{equation}
Then,
$$
\int^{t_{2}}_{t_{1}} \frac{\Box f}{\Box t}(t)\, dt=f(t_{2})-f(t_{1}).
$$
\end{theorem}


\section{The Non-Differentiable Embedding with Time Delays}
\label{sec:cNT}

Given two operators $A$ and $B$, we use the notations
$$
(A\cdot B)(y)=A(y) B(y) \quad \text{ and }  (A\circ B)(y)=A(B(y)).
$$

\begin{definition}[The $k$th scale derivative]
Let $k \in \mathbb{N}$. The operator $\square^{k}$
is defined by
\begin{equation}
\label{eq:notat}
\square^{k}
= \frac{\square^{k}}{\square t^{k}}
= \frac{\square}{\square t}\circ \cdots \circ \frac{\square}{\square t},
\end{equation}
where $\frac{\square}{\square t}$ appears exactly $k$ times
on the right-hand side of \eqref{eq:notat}.
\end{definition}

\begin{definition}[The backward shift operator $\rho^\tau$]
Given $\tau > 0$, the backward shift operator $\rho^\tau$ is defined by
$\rho^\tau(t) = t - \tau$.
\end{definition}

\begin{definition}[The operators ${[\cdot]^{\square^k}_{\tau}}$ and ${[\cdot]^{k}_{\tau}}$]
Let $\tau>0$, $k \in \mathbb{N}$. For convenience,
we introduce the operators ${[\cdot]^{\square^k}_{\tau}}$ and ${[\cdot]^{k}_{\tau}}$ by
$$
[y]^{\square^k}_\tau(t)=\left(t,
y(t),\square y(t),\ldots,\square^{k}y(t),(y \circ \rho^\tau)(t),
(\square y\circ \rho^\tau)(t),\ldots, (\square^{k}y\circ \rho^\tau)(t) \right)
$$
and
$$
[y]^{k}_\tau(t)=\left(t,
y(t),y'(t),\ldots,y^{(k)}(t),(y\circ \rho^\tau)(t), (y'\circ \rho^\tau)(t),
\ldots,(y^{(k)}\circ \rho^\tau)(t) \right).
$$
When $k = 1$, we omit $k$ and use
$\dot{y}$ to denote the derivative $y'$, that is:
$$
[y]_{\tau}^{\square}(t)=(t,y(t),\square y(t),y(t-\tau),\square y(t-\tau))
$$
and
$$
[y]_{\tau}(t)=(t,y(t),\dot{y}(t),y(t-\tau),\dot{y}(t-\tau)).
$$
\end{definition}

Given a function $f :\mathbb{R} \times \left(\mathbb{C}^d\right)^{2(k+1)}
\rightarrow \mathbb{C}$, we denote by $F^{k,\tau}$
the corresponding \emph{evaluation operator} defined by
$F^{k,\tau} = f[\cdot]^{k}_\tau$, that is,
\begin{equation*}
F^{k,\tau} :
\begin{array}{lll}
\mathcal{C}^{0}\left(I, \mathbb{C}^{d}\right) & \longrightarrow &
\mathcal{C}^{0}\left(I,
\mathbb{C}\right) \\
y & \longmapsto & t \mapsto f[y]^{k}_\tau(t)\, .
\end{array}
\end{equation*}

Let $\mathbf{f}=\{ f_i \}_{i=0,\ldots ,n}$  and
$\mathbf{g}=\{ g_i \}_{i=0,\ldots ,n}$ be a finite family of
functions $f_i, g_i :\mathbb{R} \times \left(\mathbb{C}^d\right)^{2(k+1)}
\rightarrow \mathbb{C}$, and $F^{k,\tau}_i$ and $G^{k,\tau}_i$, $i=1,\ldots ,n$,
be the corresponding family of evaluation operators. We denote by
$\mathrm{O}^{k,\tau}_{\mathbf{f}, \mathbf{g}}$ the differential operator
\begin{equation}
\label{formop}
\mathrm{O}^{k,\tau}_{\mathbf{f},\mathbf{g}}
=\sum_{i=0}^n F^{k,\tau}_i \cdot  \left({d^i \over dt^i }
\circ G^{k,\tau}_i\right),
\end{equation}
with the convention that $ \left (  {d\over dt} \right )^0$
is the identity mapping on $\mathbb{C}$.
As before, we omit $k$ when $k = 1$:
$\mathrm{O}^{\tau}_{\mathbf{f},\mathbf{g}}
= \mathrm{O}^{1,\tau}_{\mathbf{f},\mathbf{g}}$.

\begin{definition}[Non-differentiable embedding of operators with time delay]
\label{ndeo}
The non-differentiable embedding of \eqref{formop},
denoted by $\mbox{\rm Emb}_{\square}
\left(\mathrm{O}^{k,\tau}_{\mathbf{f},\mathbf{g}}\right)$, is given by
\begin{equation*}
\mbox{\rm Emb}_{\square}\left(\mathrm{O}^{k,\tau}_{\mathbf{f},\mathbf{g}}\right)
= \sum_{i=0}^n F^{\square^k,\tau}_i \cdot  \left({\square^{i} \over \square t^i }
\circ G^{\square^k,\tau}_i\right),
\end{equation*}
$F^{\square^k,\tau}_i = \mbox{\rm Emb}_{\square}(F^{k,\tau}_i)
= f_i[\cdot]^{\square^k}_\tau$,
${\square^{i} \over \square t^i }
= \mbox{\rm Emb}_{\square}\left({d^i \over dt^i }\right)$,
$G^{\square^k,\tau}_i = \mbox{\rm Emb}_{\square}(G^{k,\tau}_i)
= g_i[\cdot]^{\square^k}_\tau$.
\end{definition}


\subsection{Embedding of Variational Problems with Time Delay}
\label{sec:CV}

The fundamental problem of the calculus of variations with time delay
is to minimize
\begin{equation}
\label{P}
I^{\tau}[q] = \int_{t_{1}}^{t_{2}}
L\left(t,q(t),\dot{q}(t),q(t-\tau),\dot{q}(t-\tau)\right) dt
\end{equation}
subject to
$q(t)=\delta(t)$, $t\in[t_{1}-\tau,t_{1}]$, and $q(t_2) = q_2$,
where $t_{1}< t_{2}$ are fixed in $\mathbb{R}$,
$\tau$ is a given positive real number such that $\tau<t_{2}-t_{1}$,
$\delta$ is a given piecewise smooth function, and $q_2 \in \mathbb{R}^d$.
We assume that admissible functions $q$ are such that both
$q$ and $q\circ \rho^\tau$ belong to $\mathcal{C}^{1}\left(I,\mathbb{R}^{d}\right)$.
Note that, with our notations, \eqref{P} can be written as
$$
I^{\tau}[q] = \int_{t_{1}}^{t_{2}} L[q]_\tau(t) dt.
$$

A variation of $q\in \mathcal{C}^{1}\left(I,
\mathbb{R}^{d}\right)$ is another function of $
\mathcal{C}^{1}\left(I, \mathbb{R}^{d}\right)$ of the form
$q + \varepsilon h$ with $\varepsilon$ a small number
and $h \in \mathcal{C}^{1}\left(I,\mathbb{R}^{d}\right)$
such that $h(t)=0$ for $t \in[t_{1}-\tau,t_{1}] \cup  \{t_2\}$.

\begin{definition}[Extremal]
We say that $q$ is an \emph{extremal} of funcional \eqref{P} if
$$
\frac{d}{d\varepsilon}\left.I^{\tau}[y + \varepsilon h]\right|_{\varepsilon = 0}=0
$$
for any $h\in \mathcal{C}^{1}\left(I,\mathbb{R}^{d}\right)$.
\end{definition}

A first idea to obtain a non-differentiable Lagrangian system
with time delays is to embed directly the classical
Euler--Lagrange equations with time delays.

\begin{theorem}[Euler--Lagrange equations with time delay \cite{CD:Agrawal:1997,DH:1968}]
A function $q\in \mathcal{C}^{1}\left(I, \mathbb{R}^{d}\right)$ is
an extremal of \eqref{P} if and only if
\begin{equation}
\label{EL1}
\begin{cases}
\frac{d}{dt}\left[L_{\dot{q}}[q]_{\tau}(t)+
L_{\dot{q}_{\tau}}[q]_{\tau}(t+\tau)\right]\\
\qquad=L_{q}[q]_{\tau}(t)+L_{q_{\tau}}[q]_{\tau}(t+\tau)\,,
\quad t_{1}\leq t\leq t_{2}-\tau,\\  \tag{EL}
\frac{d}{dt}L_{\dot{q}}[q]_{\tau}(t) =L_{q}[q]_{\tau}(t)\,,
\quad t_{2}-\tau\leq t\leq t_{2},
\end{cases}
\end{equation}
where $L_{\xi}(t,q,\dot{q},q_\tau,\dot{q}_\tau)$
denotes the partial derivative of
$L(t,q,\dot{q},q_\tau,\dot{q}_\tau)$ with respect to
$\xi \in \{q,\dot{q},q_\tau,\dot{q}_\tau\}$.
\end{theorem}

The following theorem gives the non-differentiable embedding
of the Euler--Lagrange equations with time delay \eqref{EL1}.
By $\mathcal{C}^{n}_{\square}\left(I, \mathbb{X}^{d}\right)$
we denote the set of functions $q$ such that both $q$
and $q \circ \rho^\tau$ belong to
$\mathcal{C}^{0}\left(I, \mathbb{X}^{d}\right)$ as well as
$\square^{i}q$ and $(\square^{i}q) \circ \rho^\tau$ for
all $i=1,\ldots,n$.

\begin{theorem}
Let the Lagrangian $L$ be a $\mathcal{C}^{1}_{\square}\left(I, \mathbb{R}^{d}\right)$-function
with respect to all its arguments, holomorphic with respect to
$\dot{q}(t)$ and $\dot{q}(t-\tau)$, and real when
$\dot{q}(t)$ and $\dot{q}(t-\tau)$ belong to $\mathbb{R}^{d}$.
The non-differentiable embedded
\emph{Euler--Lagrange equations with time delay}
associated to $L$ are given by
\begin{equation}
\label{EL2}
\begin{cases}
\frac{\square}{\square t}\left[L_{\square
q}[q]_{\tau}^{\square}(t)+ L_{\square
q_{\tau}}[q]_{\tau}^{\square}(t+\tau)\right]\\
\qquad=L_{q}[q]_{\tau}^{\square}(t)
+L_{q_{\tau}}[q]_{\tau}^{\square}(t+\tau),\,\,\, t_{1}\leq t
\leq t_{2}-\tau,\\ \tag{$\square EL$}
\frac{\square}{\square t} L_{\square q}[q]_{\tau}^{\square}(t)
=L_{q}[q]_{\tau}^{\square}(t),\,\,\, t_{2}-\tau\leq t\leq t_{2}\,.
\end{cases}
\end{equation}
\end{theorem}

\begin{proof}
The Euler--Lagrange equations \eqref{EL1} can be written in the
equivalent form
$$
\mathrm{O}^{\tau}_{\mathbf{f},\mathbf{g}}(q)(t)=0,
\quad t\in [t_{1}, t_{2}],
$$
with $\mathbf{f}$ and $\mathbf{g}$ given by
\begin{equation*}
\mathbf{f}[q]_{\tau}(t)
=
\begin{cases}
\left(-L_{q}[q]_{\tau}(t)-L_{q_{\tau}}[q]_{\tau}(t+\tau), 1\right)
& \text{ if } t\in [t_{1},t_{2}-\tau]\\
\left(-L_{q}[q]_{\tau}(t), 1\right) & \text{ if } t\in [t_{2}-\tau, t_{2}]
\end{cases}
\end{equation*}
and
\begin{equation*}
\mathbf{g}[q]_{\tau}(t) =
\begin{cases}
\left(1,L_{\dot{q}}[q]_{\tau}(t)+
L_{\dot{q_{\tau}}}[q]_{\tau}(t+\tau) \right)
& \text{ if } t\in [t_{1},t_{2}-\tau]\\
\left(1,L_{\dot{q}}[q]_{\tau}(t) \right)
& \text{ if } t\in [t_{2}-\tau, t_{2}].
\end{cases}
\end{equation*}
The intended conclusion follows now
by a direct application of Definition~\ref{ndeo}.
\end{proof}

\begin{remark}
The Euler--Lagrange equations \eqref{EL2} reduce to \eqref{EL1}
when the functions are differentiable.
\end{remark}

Another approach to obtain a non-differentiable Lagrangian system
with time delays is to embed the Lagrangian,
and then to apply the least-action principle.
The non-differentiable embedding of functional \eqref{P}
is given by
\begin{equation}
\label{P1}
I_{\square}^{\tau}[q] = \int_{t_{1}}^{t_{2}}
L\left(t,q(t),\square q(t),q(t-\tau),\square q(t-\tau)\right) dt
= \int_{t_{1}}^{t_{2}} L[y]_{\tau}^{\square}(t) dt.
\end{equation}
In contrast with the original problem, the admissible functions
$q$ are now not necessarily differentiable: admissible functions
$q$ for \eqref{P1} are those such that
$q \in \mathcal{C}^{1}_{\square}\left(I,\mathbb{R}^{d}\right)$.

Let $\alpha,\beta,\varepsilon \in \mathbb{R}$ be such that
$0<\alpha,\beta<1$, $\alpha+\beta>1$ and $|\varepsilon| \ll 1$.
A variation of $q\in H^{\alpha}\left(I, \mathbb{R}^{d}\right)$ is
another function of $ H^{\alpha}\left(I, \mathbb{R}^{d}\right)$ of
the form $q+ \varepsilon h$, with $h \in H^{\beta}\left(I,
\mathbb{R}^{d}\right)$, such that $h(t)=0$ for
$t \in[t_{1}-\tau,t_{1}] \cup \{t_2\}$.

\begin{definition}[Non-differentiable extremal]
\label{nde}
We say that $q$ is a \emph{non-differentiable extremal}
of funcional \eqref{P1} if
\begin{equation}
\label{diff}
\frac{d}{d\varepsilon}\left.I_{\square}^{\tau}[y + \varepsilon
h]\right|_{\varepsilon = 0}=0
\end{equation}
for any $h\in H^{\beta}\left(I,\mathbb{R}^{d}\right)$.
\end{definition}

As in the classical case, the least-action principle has here
its own meaning, i.e., we seek the non-differentiable extremals
of  funcional \eqref{P1} to determine the dynamics
of a non-differentiable dynamical system.
The next theorem gives the Euler--Lagrange
equations with time delay obtained from
the non-differentiable least-action principle.

\begin{theorem}
\label{ELDT}
Let $0<\alpha,\beta<1$ with $\alpha+\beta>1$.
If $q\in H^{\alpha}\left(I,\mathbb{R}^{d}\right)$
satisfies $\square q\in H^{\alpha}\left(I,
\mathbb{R}^{d}\right)$ and
$\left(L_{\square q}[q]_{\tau}^{\square}(t)+L_{\square
q_{\tau}}[q]_{\tau}^{\square}(t)\right)\cdot h(t)$
satisfies condition \eqref{nec_condition} for all $h\in
H^{\beta}\left(I, \mathbb{R}^{d}\right)$, then function $q$ is a
non-differentiable extremal of $I_{\square}^{\tau}$ if and only if
\begin{equation}
\label{EL3}
\begin{cases}
\frac{\square}{\square t}\left(L_{\square
q}[q]_{\tau}^{\square}(t)+ L_{\square
q_{\tau}}[q]_{\tau}^{\square}(t+\tau)\right)\\
\qquad\quad
= L_{q}[q]_{\tau}^{\square}(t)+L_{q_{\tau}}[q]_{\tau}^{\square}(t+\tau),
\quad t_{1}\leq t\leq t_{2}-\tau,\\ \tag{$EL_{\square LAP}$}
\frac{\square}{\square t} L_{\square q}[q]_{\tau}^{\square}(t)
=L_{q}[q]_{\tau}^{\square}(t),\quad t_{2}-\tau\leq t\leq t_{2}.
\end{cases}
\end{equation}
\end{theorem}

\begin{proof}
Condition \eqref{diff} gives
\begin{multline}
\label{II}
\int_{t_{1}}^{t_{2}} \left(L_{q}[q]_{\tau}^{\square}(t)\cdot h(t)
+ L_{\square q}[q]_{\tau}^{\square}(t)\cdot \square h(t)\right)dt\\
+\int_{t_{1}}^{t_{2}}
\left(L_{q_{\tau}}[q]_{\tau}^{\square}(t)\cdot
h(t-\tau)+L_{\square q_{\tau}}[q]_{\tau}^{\square}(t)\cdot \square
h(t-\tau)\right) dt=0\,.
\end{multline}
By the linear change of variables $t=s+\tau$ in the last integral
of \eqref{II}, and having in mind the fact that $h(t)=0$ on
$[t_1-\tau,t_1]$, equation \eqref{II} becomes
\begin{multline}
\label{I}
\int_{t_{1}}^{t_{2}-\tau}\left[\left(L_{q}[q]_{\tau}^{\square}(t)
+L_{q_{\tau}}[q]_{\tau}^{\square}(t+\tau)\right)\cdot h(t)\right. \\
\left. +\left(L_{\square q}[q]_{\tau}^{\square}(t)+L_{\square
q_{\tau}}[q]_{\tau}^{\square}(t+\tau)\right)\cdot \square h(t)\right]dt\\
+\int_{t_{2}-\tau}^{t_{2}}\left(L_{q}[q]_{\tau}^{\square}(t)\cdot
h(t)+L_{\square q}[q]_{\tau}^{\square}(t)\cdot \square
h(t)\right)dt=0\,.
\end{multline}
Using the hypotheses of the theorem, we obtain
from Theorem~\ref{theo:mult} that
\begin{multline}
\label{III}
\int_{t_{1}}^{t_{2}-\tau}\left(L_{\square
q}[q]_{\tau}^{\square}(t)+L_{\square
q_{\tau}}[q]_{\tau}^{\square}(t+\tau)\right)\cdot \square
h(t)dt\\=\int_{t_{1}}^{t_{2}-\tau}\frac{\square}{\square
t}\left\{\left(L_{\square q}[q]_{\tau}^{\square}(t)+L_{\square
q_{\tau}}[q]_{\tau}^{\square}(t+\tau)\right)\cdot h(t)\right\}dt\\
-\int_{t_{1}}^{t_{2}-\tau}\frac{\square}{\square
t}\left(L_{\square q}[q]_{\tau}^{\square}(t)+L_{\square
q_{\tau}}[q]_{\tau}^{\square}(t+\tau)\right)\cdot h(t)dt
\end{multline}
and
\begin{multline}
\label{IIII}
\int_{t_{2}-\tau}^{t_{2}}L_{\square q}[q]_{\tau}^{\square}(t)\cdot
\square h(t) dt\\=\int_{t_{2}-\tau}^{t_{2}}\frac{\square}{\square
t}\left(L_{\square q}[q]_{\tau}^{\square}(t)\cdot
h(t)\right)dt-\int_{t_{2}-\tau}^{t_{2}}\frac{\square}{\square
t}\left(L_{\square q}[q]_{\tau}^{\square}(t)\right)\cdot h(t)dt\,.
\end{multline}
Because $\left(L_{\square q}[q]_{\tau}^{\square}(t)+L_{\square
q_{\tau}}[q]_{\tau}^{\square}(t+\tau)\right)\cdot h(t)$
satisfies \eqref{nec_condition} for all $h\in H^{\beta}\left(I,
\mathbb{R}^{d}\right)$, using Theorem~\ref{Barrow} and replacing the quantities of
\eqref{III} and \eqref{IIII} into \eqref{I}, we obtain
\begin{equation*}
\begin{split}
\int_{t_{1}}^{t_{2}-\tau}&\left[L_{q}[q]_{\tau}^{\square}(t)
+L_{q_{\tau}}[q]_{\tau}^{\square}(t+\tau)
-\frac{\square}{\square t}\left(L_{\square q}[q]_{\tau}^{\square}(t)
+ L_{\square q_{\tau}}[q]_{\tau}^{\square}(t+\tau)\right)\right]\cdot h(t)dt\\
& +\left.\left(L_{\square q}[q]_{\tau}^{\square}(t)+L_{\square q_{\tau}}
[q]_{\tau}^{\square}(t+\tau)\right)\cdot h(t)\right|^{t_2-\tau}_{t_1}\\
&+ \int_{t_{2}-\tau}^{t_{2}}\left[L_{q}[q]_{\tau}^{\square}(t)
-\frac{\square}{\square t}\left(L_{\square q}[q]_{\tau}^{\square}(t)\right)\right]
\cdot h(t)dt+\left.L_{\square q}[q]_{\tau}^{\square}(t)\cdot h(t)\right|^{t_2}_{t_2-\tau}=0\,.
\end{split}
\end{equation*}
The Euler--Lagrange equations with time delay \eqref{EL3} are obtained using
the fundamental lemma of the calculus of variations (see, e.g., \cite{MR0160139}).
\end{proof}

To summarize, the dynamics of a non-differentiable Lagrangian
system with time delay are determined by the Euler--Lagrange equations
\eqref{EL3}, via the $\square$-least-action principle, respecting the
principle of coherence: \eqref{EL3} coincide with \eqref{EL2}.


\subsection{Embedding of Optimal Control Problems with Time Delay}
\label{sec:OP}

In Section~\ref{sec:CV} we studied the non-differentiable
variational embedding in presence of time delays.
Now, we give a scale Hamiltonian embedding for more general scale problems
of optimal control with delay time. Following \cite{DH:1968,MyID:231},
we define the optimal control problem with time delay as follows:
\begin{gather}
\label{Pond1}
I^{\tau}[q,u] = \int_{t_1}^{t_2}
L\left(t,q(t),q(t-\tau),u(t),u(t-\tau)\right) dt \longrightarrow \min \, , \\
\dot{q}(t)=\varphi\left(t,q(t),q(t-\tau),u(t),u(t-\tau)\right) \label{ci}\, ,
\end{gather}
under given boundary conditions
\begin{equation}
\label{ic}
q(t)=\delta(t), \quad t\in[t_{1}-\tau,t_{1}]\,,
\quad q(t_2) = q_2,
\end{equation}
where $q \in \mathcal{C}^{1}\left(I, \mathbb{R}^{d}\right)$,
$u \in \mathcal{C}^{0}\left(I, \mathbb{R}^{d}\right)$,
the Lagrangian $L:I \times \mathbb{R}^{2d}
\times \mathbb{R}^{2m} \rightarrow \mathbb{R}$
and the velocity vector $\varphi : I \times \mathbb{R}^{2d}
\times \mathbb{R}^{2m} \rightarrow \mathbb{R}^d$
are assumed to be $C^{1}$-functions
with respect to all its arguments. Similarly as before,
we assume that $\delta$ is a given piecewise smooth function
and $q_2$ a given vector in $\mathbb{R}^d$.

\begin{definition}
We introduce the operators $[\cdot,\cdot]_{\tau}$,
$[\cdot,\cdot,\cdot]_{\tau}$, $[\cdot,\cdot]_{\tau}^{\square}$
and $[\cdot,\cdot]_{\tau}^{\square}$ by:
\begin{enumerate}
\item $[q,u]_{\tau}(t)=(t,q(t),q(t-\tau),u(t),u(t-\tau))$,
where $q \in \mathcal{C}^{1}\left(I, \mathbb{R}^{d}\right)$
and $u \in \mathcal{C}^{0}\left(I,\mathbb{R}^{d}\right)$;
\item $[q,u,p]_{\tau}(t)=(t,q(t),q(t-\tau),u(t),u(t-\tau),p(t))$,
where $q, p \in \mathcal{C}^{1}\left(I, \mathbb{R}^{d}\right)$
and $u \in \mathcal{C}^{0}\left(I,\mathbb{R}^{d}\right)$;
\item $[q,u]_{\tau}^{\square}(t)=(t,q(t),q(t-\tau),u(t),u(t-\tau))$,
where $q \in H^{\alpha}\left(I, \mathbb{R}^{d}\right)$ and
$u \in H^{\alpha}\left(I, \mathbb{C}^{d}\right)$;
\item $[q,u,p]_{\tau}^{\square}(t)=(t,q(t),q(t-\tau),u(t),u(t-\tau),p(t))$,
where $q \in H^{\alpha}\left(I, \mathbb{R}^{d}\right)$ and
$u, p \in H^{\alpha}\left(I, \mathbb{C}^{d}\right)$.
\end{enumerate}
\end{definition}

\begin{theorem}[\cite{DH:1968,MyID:231}]
\label{theo:pmpnd1}
If $(q,u)$ is a minimizer to
problem \eqref{Pond1}--\eqref{ic}, then there exists a
co-vector function $p \in \mathcal{C}^{1}\left(I,
\mathbb{R}^{d}\right)$ such that the following conditions hold:
\begin{itemize}
\item \emph{the Hamiltonian systems}
\begin{equation}
\label{eq:Hamnd1}
\begin{cases}
\dot{q}(t) = H_{p}[q,u,p]_{\tau}(t) \, , \\
\dot{p}(t)
= -H_{q}[q,u,p]_{\tau}(t)-H_{q_{\tau}}[q,u,p]_{\tau}(t+\tau)\,,
\end{cases}
\end{equation}
for $t_{1}\leq t\leq t_{2}-\tau$, and
\begin{equation}
\label{eq:Hamnd2}
\begin{cases}
\dot{q}(t) = H_{p}[q,u,p]_{\tau}(t) \, , \\
\dot{p}(t) = -H_{q}[q,u,p]_{\tau}(t) \, ,
\end{cases}
\end{equation}
for $t_{2}-\tau\leq t\leq t_{2}$;

\item \emph{the stationary conditions}
\begin{equation}
\label{eq:CE1}
H_{u}[q,u,p]_{\tau}(t)+H_{u_{\tau}}[q,u,p]_{\tau}(t+\tau)= 0\,,
\end{equation}
for $t_{1}\leq t\leq t_{2}-\tau$, and
\begin{equation}
\label{eq:CE12}
H_{u}[q,u,p]_{\tau}(t)= 0,
\end{equation}
for $t_{2}-\tau\leq t\leq t_{2}$;
\end{itemize}
where the Hamiltonian $H$ is defined by
$H[q,u,p]_{\tau}(t)= L[q,u]_{\tau}(t)+p(t) \cdot \varphi[q,u]_{\tau}(t)$.
\end{theorem}

\begin{lemma}
Let $H[q,u,p]_{\tau}^{\square}(t)= L[q,u]_{\tau}^{\square}(t)+p(t)
\cdot \varphi[q,u]_{\tau}^{\square}(t)$.
The embedding of the Hamiltonian systems \eqref{eq:Hamnd1}
and \eqref{eq:Hamnd2} are given, respectively, by
\begin{equation}
\label{eq:Hamnd3}
\begin{cases}
\square q(t)=H_{p}[q,u,p]_{\tau}^{\square}(t) \, , \\
\square p(t)=-H_{q}[q,u,p]_{\tau}^{\square}(t)
-H_{q_{\tau}}[q,u,p]_{\tau}^{\square}(t+\tau),
\end{cases}
\end{equation}
for $t_{1}\leq t\leq t_{2}-\tau$, and by
\begin{equation}
\label{eq:Hamnd4}
\begin{cases}
\square q(t)= H_{p}[q,u,p]_{\tau}^{\square}(t) \, , \\
\square p(t) = -H_{q}[q,u,p]_{\tau}^{\square}(t) \, ,
\end{cases}
\end{equation}
for $t_{2}-\tau\leq t\leq t_{2}$; the embedding of the stationary conditions
\eqref{eq:CE1} and \eqref{eq:CE12} are given, respectively, by
\begin{equation}
\label{eq:CE2}
H_{u}[q,u,p]_{\tau}^{\square}(t)
+H_{u_{\tau}}[q,u,p]_{\tau}^{\square}(t+\tau)= 0,
\end{equation}
for $t_{1}\leq t\leq t_{2}-\tau$, and by
\begin{equation}
\label{eq:CE22}
H_{u}[q,u,p]_{\tau}^{\square}(t)= 0\, ,
\end{equation}
for $t_{2}-\tau\leq t\leq t_{2}$.
\end{lemma}

\begin{definition}
We call systems \eqref{eq:Hamnd3} and \eqref{eq:Hamnd4}
\emph{the scale Hamiltonian systems with delay},
while to \eqref{eq:CE2} and \eqref{eq:CE22} we
call \emph{stationary conditions with delay}.
\end{definition}

\begin{lemma}
The embedding of \eqref{Pond1}--\eqref{ci} is given by
\begin{gather}
\label{Pond11}
I^{\tau}_{\square}[q, u] = \int_{t_1}^{t_2}
L\left(t,q(t),q(t-\tau),u(t),u(t-\tau)\right) dt \longrightarrow \min \, , \\
\square q(t)=\varphi\left(t,q(t),q(t-\tau),u(t),u(t-\tau)\right), \label{ci1}
\end{gather}
where $q, q \circ \rho^\tau \in H^{\alpha}\left(I,
\mathbb{R}^{d}\right)$ and $u, u \circ \rho^\tau
\in H^{\alpha}\left(I,\mathbb{C}^{d}\right)$.
\end{lemma}

Theorem~\ref{th:pmp} generalizes Theorem~\ref{theo:pmpnd1} for
non-differentiable optimal control problems with time delay.

\begin{theorem}
\label{th:pmp}
Let $0<\alpha$, $\beta<1$ with $\alpha+\beta>1$.
Assume that $q\in H^{\alpha}\left(I, \mathbb{R}^{d}\right)$
satisfies $\square q\in H^{\alpha}\left(I, \mathbb{R}^{d}\right)$ and
$\left(L_{\square q}[q]_{\square}^{\tau}(t)
+L_{\square q_{\tau}}[q]_{\square}^{\tau}(t)\right)\cdot h(t)$
satisfies condition \eqref{nec_condition} for all
$h\in H^{\beta}\left(I, \mathbb{R}^{d}\right)$.
If $(q, u)$ is a minimizer to problem
\eqref{Pond11}--\eqref{ci1} subject to
given boundary conditions \eqref{ic}, then there
exists a co-vector function $p \in H^{\alpha}\left(I,
\mathbb{C}^{d}\right)$ such that the following conditions hold:
\begin{itemize}
\item the scale Hamiltonian systems with delay \eqref{eq:Hamnd3}
and \eqref{eq:Hamnd4};

\item the stationary conditions with delay \eqref{eq:CE2}
and \eqref{eq:CE22}.
\end{itemize}
\end{theorem}

\begin{proof}
We prove the theorem  only in the interval $t_{1}\leq t\leq t_{2}-\tau$
(the reasoning is similar in interval $t_{2}-\tau\leq t\leq t_{2}$).
Using the Lagrange multiplier rule, \eqref{Pond11}--\eqref{ci1}
is equivalent to minimize the augmented functional $J^{\tau}_{\square}[q,u,p]$
defined by
\begin{equation}
\label{eq:pcond}
J^{\tau}_{\square}[q,u,p]
= \int_{t_{1}}^{t_{2}} \left[
H\left(t,q(t),q(t-\tau),u(t),u(t-\tau),p(t)\right)-p(t)\cdot\square
q(t)\right]dt.
\end{equation}
The necessary optimality conditions \eqref{eq:Hamnd3} and
\eqref{eq:CE2} are obtained from the Euler--Lagrange equations
\eqref{EL3} applied to functional \eqref{eq:pcond}:
\begin{multline*}
\begin{cases}
\frac{\square}{\square t} \left(\mathbb{L}_{\square
q}[q,u,p]_{\square}^{\tau}(t)+\mathbb{L}_{\square
q_{\tau}}[q,u,p]_{\square}^{\tau}(t+\tau)\right)\\
\qquad\qquad=\mathbb{L}_{q}[q,u,p]_{\square}^{\tau}(t)
+\mathbb{L}_{ q_{\tau}}[q,u,p]_{\square}^{\tau}(t+\tau)\\
\frac{\square}{\square t} \left(\mathbb{L}_{\square
u}[q,u,p]_{\square}^{\tau}(t)+\mathbb{L}_{\square
u_{\tau}}[q,u,p]_{\square}^{\tau}(t+\tau)\right)\\
\qquad\qquad =\mathbb{L}_{u}[q,u,p]_{\square}^{\tau}(t)
+\mathbb{L}_{ u_{\tau}}[q,u,p]_{\square}^{\tau}(t+\tau)\; \\
\frac{\square}{\square t} \left(\mathbb{L}_{\square p}[q,u,p]_{\square}^{\tau}(t)
+\mathbb{L}_{\square p_{\tau}}[q,u,p]_{\square}^{\tau}(t+\tau)\right)\\
\qquad\qquad = \mathbb{L}_{ p}[q,u,p]_{\square}^{\tau}(t)
+\mathbb{L}_{p_{\tau}}[q,u,p]_{\square}^{\tau}(t+\tau)
\end{cases}\\
\Leftrightarrow\;
\begin{cases}
\square p(t)=-H_{q}[q,u,p]_{\square}^{\tau}(t)-H_{q_{\tau}}[q,u,p]_{\square}^{\tau}(t+\tau)\\
0=H_{u}[q,u,p]_{\square}^{\tau}(t)+H_{u_{\tau}}[q,u,p]_{\square}^{\tau}(t+\tau)\\
0=-\square q(t)+H_{p}[q,u,p]_{\square}^{\tau}(t),
\end{cases}
\end{multline*}
where $\mathbb{L}[q,u,p]_{\square}^{\tau}(t)=H[q,u,p]_{\square}^{\tau}(t)-p(t)
\cdot\square q(t)$.
\end{proof}

\begin{remark}
In the differentiable case Theorem~\ref{th:pmp}
reduces to Theorem~\ref{theo:pmpnd1}.
\end{remark}

\begin{remark}
The first equations in the scale Hamiltonian system with delay
\eqref{eq:Hamnd3} and \eqref{eq:Hamnd4} are nothing else than the
scale control system \eqref{ci1}.
\end{remark}

\begin{remark}
In classical mechanics, $p$ is called the \emph{generalized
momentum}. In the language of optimal control,
$p$ is known as the adjoint variable \cite{CD:MR29:3316b}.
\end{remark}

\begin{definition}
\label{scale:Pont:Ext}
A triplet $(q,u,p)$ satisfying the conditions
of Theorem~\ref{th:pmp} will be called a
\emph{scale Pontryagin extremal}.
\end{definition}

\begin{remark}
\label{rem:cp:CV:EP:EEL}
If $\varphi\left(t,q,q_\tau,u,u_\tau\right) = u$, then
Theorem~\ref{th:pmp} reduces to Theorem~\ref{ELDT}.
Let us verify this in the interval $t_{1}\leq t\leq t_{2}-\tau$
(the procedure is similar for $t_{2}-\tau\leq t\leq t_{2}$).
The stationary condition \eqref{eq:CE2} gives
$p(t) = L_{u}[q]_{\tau}^{\square}(t)+L_{u_{\tau}}[q]_{\tau}^{\square}(t+\tau)$
and the second equation in the scale Hamiltonian
system with delay \eqref{eq:Hamnd3} gives
$\square p(t) = L_{q}[q]_{\tau}^{\square}(t)+L_{q_{\tau}}[q]_{\tau}^{\square}(t)$.
Comparing both equalities, one obtains the non-differentiable
Euler--Lagrange equations with time delay \eqref{EL3}. In other
words, the scale Pontryagin extremals
(Definition~\ref{scale:Pont:Ext}) are a generalization of the
non-differentiable Euler--Lagrange extremals (Definition~\ref{nde}).
\end{remark}

We conclude from Theorem~\ref{th:pmp} that the coherence principle
is also respected for non-differentiable optimal control problems with time delay.


\section*{Acknowledgments}

This work was supported by {\it FEDER} funds through
{\it COMPETE} --- Operational Programme Factors of Competitiveness
(``Programa Operacional Factores de Competitividade'')
and by Portuguese funds through the
{\it Center for Research and Development
in Mathematics and Applications} (University of Aveiro)
and the Portuguese Foundation for Science and Technology
(``FCT --- Funda\c{c}\~{a}o para a Ci\^{e}ncia e a Tecnologia''),
within project PEst-C/MAT/UI4106/2011
with COMPETE number FCOMP-01-0124-FEDER-022690.
The first author was also supported by the FCT post-doc
fellowship SFRH/BPD/51455/2011, from program {\it Ci\^{e}ncia Global}.



\end{document}